\documentclass[11 pt]{amsart}
\usepackage{graphicx}
\usepackage{amsmath}
\usepackage{mathtools}
\usepackage
{amstext, amscd, setspace, tikz-cd,mathtools, enumerate, graphics, latexsym, siunitx}

\usepackage{pst-node}
\usepackage{tikz-cd}

\usepackage{amssymb}
\usepackage{hyperref}

\usepackage{PTSerif}
\usepackage{graphicx}

\usepackage[cmtip,all]{xy}
\newcommand{\properideal}{%
\mathrel{\ooalign{$\lneq$\cr\raise.22ex\hbox{$\lhd$}\cr}}}

\newcommand{\rank}{\operatorname{rank}}

\theoremstyle{definition}
\numberwithin{equation}{subsubsection}

\newtheorem{theorem}{Theorem}[section]
\newtheorem{corollary}{Corollary}[theorem]

\newtheorem{lemma}[theorem]{Lemma}
\newtheorem{proposition}[theorem]{Proposition}

\newtheorem{definition}[theorem]{Definition}

\newtheorem{hypothesis}[theorem]{Hypothesis}

\numberwithin{equation}{section}

\title[The Regular property of Invariant Rings over Regular Domains]{The Regular property of Invariant Rings over Regular Domains}

\author{Shubham Jaiswal, Tony J. Puthenpurakal}

\address{Department of Mathematics IIT Bombay, Powai, Mumbai 400 076, India.}

\email{sjaiswal@math.iitb.ac.in, tputhen@math.iitb.ac.in}


\date{\today}


\begin{document}

\begin{abstract}

The main result of this paper is a generalization of the
theorem of Chevalley-Shephard-Todd to the rings of invariants
of pseudo-reflection groups over regular domains. More precisely, let $A$ be a regular domain and let $K$ be its field of fractions. Let $G\subseteq GL_n(A)$ be a finite group. Let $G$ act linearly on $A[X_1,X_2,\dots, X_n]$ (fixing $A$). Assume that $|G|$ is invertible in $A$. We prove that $G\subseteq GL_n(K)$ is generated by pseudo-reflections if and only if $(A[X_1,X_2,\dots, X_n])^G$ is regular.\end{abstract}

\maketitle

\section{Introduction}

We begin by recalling the notion of pseudo-reflections.

\begin{definition}[\S 7.1, \cite{benson1993polynomial}]

Let $K$ be a field. An element $\sigma\in GL_n(K)$ is called a pseudo-reflection if $ord(\sigma)$ is finite and $dim(ker(\sigma-1))=n-1$, i.e. whose fixed points have codimension one.

\end{definition}

We have the famous Chevalley-Shephard-Todd Theorem (Theorem 7.2.1 in \cite{benson1993polynomial}) which states:

\begin{theorem}\label{CST}

Let $K$ be a field. Consider a finite group $G\subseteq GL_n(K)$. Consider the ring\\
$K[X_1,X_2,\dots, X_n]$ and let $G$ act linearly on the ring (fixing $K$). Assume that $|G|$ is invertible in $K$. We have $G\subseteq GL_n(K)$ is generated by pseudo-reflections if and only if $(K[X_1,X_2,\dots, X_n])^G$ is a polynomial ring over $K$.

\end{theorem}

Mundelius had generalized Theorem \ref{CST} to the rings of invariants
of pseudo-reflection groups over Dedekind domains (Theorem 3.8 in \cite{MUNDELIUS2022244}).\smallskip

In this paper we establish the following result which is a further generalization of Theorem \ref{CST} to the rings of invariants of pseudo-reflection groups over regular domains.

\begin{theorem} \label{main thm}
  Let $A$ be a regular domain and let $K$ be its field of fractions. Consider a finite group $G\subseteq GL_n(A)$. Consider the ring $A[X_1,X_2,\dots, X_n]$ and let $G$ act linearly on the ring (fixing $A$). Assume that $|G|$ is invertible in $A$. We have $G\subseteq GL_n(K)$ is generated by pseudo-reflections if and only if $(A[X_1,X_2,\dots, X_n])^G$ is regular.
\end{theorem}
\smallskip

One direction of the above theorem follows from Theorem \ref{CST}. Here's the proof.

\begin{proof}
Suppose $(A[X_1,X_2,\dots, X_n])^G$ is regular. Now $G$ acts linearly on the ring $A[X_1,X_2,\dots, X_n]$ (fixing $A$) and $|G|$ is invertible in $A$. As $K$ is the field of fractions of $A$, so $G$ also acts linearly on the ring $K[X_1,X_2,\dots, X_n]$ (fixing $K$) and $|G|$ is invertible in $K$. Now as $(K[X_1,X_2,\dots, X_n])^G$ is a localization of $(A[X_1,X_2,\dots, X_n])^G$, we have that $(K[X_1,X_2,\dots, X_n])^G$ is regular. Therefore by Theorem 5.3.3 in \cite{neusel2002invariant} and by Theorem 19.2 in \cite{matsumura1989commutative}, $(K[X_1,X_2,\dots, X_n])^G$ is a polynomial ring over the field $K$. Hence by Theorem \ref{CST}, $G\subseteq GL_n(K)$ is generated by pseudo-reflections .
\end{proof}




The subsequent sections of the paper are dedicated to proving the other direction. In Section \ref{DVR}, we prove Theorem \ref{main thm} for the case when $A$ is a DVR as Theorem \ref{thm}. This also leads to an alternate proof for the case of Dedekind domains considered by Mundelius in \cite{MUNDELIUS2022244}. In Section \ref{reg}, we prove Theorem \ref{main thm} for the case when $A$ is a regular local ring as Theorem \ref{thm reg}. Finally in Section \ref{regdom}, we complete the proof of Theorem \ref{main thm}.


\section{Proof of Theorem \ref{main thm} for DVRs}
\label{DVR}

Consider the hypothesis below.

\begin{hypothesis}
\label{hypo dvr}
         Assume that $(\mathcal{O},(\pi))$ is a DVR and $\mathcal{O}_\pi$ is the localization of $\mathcal{O}$ at the multiplicative set $\{\pi^i\}_{i}$ (which is the field of fractions of $\mathcal{O}$). Consider a finite group $G\subseteq GL_n(\mathcal{O})$. Consider the ring $\mathcal{O}[X_1,X_2,\dots, X_n]$ and assume that $G$ acts linearly on the ring (fixing $\mathcal{O}$). Assume that $|G|$ is invertible in $\mathcal{O}$ i.e. $|G|$ is invertible in $\mathcal{O}/\pi\mathcal{O}$.
\end{hypothesis}

In this section, we will prove the following.

\begin{theorem} \label{thm}
 Assume Hypothesis \ref{hypo dvr}. Consider $G\subseteq GL_n(\mathcal{O}_{\pi})$ is generated by pseudo-reflections. Then $(\mathcal{O}[X_1,X_2,\dots, X_n])^G$ is regular.
\end{theorem}

We can easily generalize Proposition 3.5 in \cite{puthenpurakal2016two} (which was for $n=2$) for any $n\in \mathbb{N}$ for our case. We provide the proof below for the sake of completeness.

\begin{proposition} \label{prop gen}
Assume Hypothesis \ref{hypo dvr}. For $n\in \mathbb{N}$, the natural map $\eta : G\rightarrow GL_n(\mathcal{O}/\pi \mathcal{O})$ is injective.\end{proposition}

\begin{proof}
    We consider elements of $GL_n(\mathcal{O})$ and $GL_n(\mathcal{O}/\pi \mathcal{O})$ as matrices. Let $W\in G$ with $\eta (W)=I$. Then $W=I+Y$ where all entries of $Y$ are in $\pi \mathcal{O}$. Now $(I+Y)^r=W^r=I$ where $r=|G|$. Thus $$rY+ {r\choose 2} Y^2+\dots +\ {r\choose {r-1}} Y^{r-1} + Y^r=0.$$ Set $$U=rI+\ {r\choose 2} Y+\dots +\ {r\choose {r-1}} Y^{r-2} + Y^{r-1}.$$

    So $YU=0$. The image of $U$ in $GL_n(\mathcal{O}/\pi \mathcal{O})$ is $rI$ which is invertible in $GL_n(\mathcal{O}/\pi \mathcal{O})$ since $r$ is invertible in $\mathcal{O}/\pi \mathcal{O}$. Consider the map $U:\mathcal{O}^n\rightarrow \mathcal{O}^n$. By Nakayama's Lemma it follows that this map is surjective. Since $\mathcal{O}^n$ is noetherian module over $\mathcal{O}$, the map is an isomorphism. Thus $U\in GL_n(\mathcal{O})$. As $YU=0$, we get $Y=0$. So $\eta$ is injective.\end{proof}

The following is an integral lemma for proving Theorem \ref{thm}.

\begin{lemma}\label{new lemma}

   Assume Hypothesis \ref{hypo dvr}. Suppose that $G$ is generated by pseudo-reflections in\\
$GL_n(\mathcal{O}_{\pi})$. Then for $n\geq 2$ and any pseudo-reflection $\sigma\in G\subseteq GL_n(\mathcal{O}_{\pi})$, there exists a basis $\{w_1,w_2,\dots w_n\}$ of $\mathcal{O}^n$ with $\sigma(w_i)=w_i$ for $1\leq i\leq n-1$ and $\sigma(w_n)=\lambda w_n$ where $\lambda$ is an $m$-th root of unity where $m=ord(\sigma)$.

\end{lemma}

\begin{proof}
Let $\{v_1, \ldots v_n \}$ be a basis of $(\mathcal{O}_\pi)^n$ such that $\sigma(v_i) = v_i$ for $1 \leq i \leq n-1$ and $\sigma(v_n) = \lambda v_n$.

Let $<\lambda>$ be the subgroup of $\mathcal{O}^{\times}$ generated by $\lambda$. Since $ord(\lambda)=ord(\sigma)$, we have $|ord(\lambda)|\mid |G|$. As $|G|$ is invertible in $\mathcal{O}/\pi \mathcal{O}$, so $|<\lambda>|$ is also invertible in $\mathcal{O}/\pi \mathcal{O}$. By Proposition \ref{prop gen}, the map $<\lambda> \rightarrow (\mathcal{O}/\pi \mathcal{O})^{\times}$ is injective. As $\lambda\neq 1$, we have that the residue of $\lambda$ in $\mathcal{O}/\pi\mathcal{O}$ is $\bar{\lambda}\neq 1$. Thus $\lambda-1 \not \in \pi \mathcal{O}$. So $\lambda -1\in \mathcal{O}^{\times}$.\smallskip

Consider the endomorphism $\epsilon  = (\lambda - 1)^{-1}(\sigma - id) \colon \mathcal{O}^n \rightarrow \mathcal{O}^n$. By considering $\epsilon_\pi$ it follows that $\epsilon$ is an idempotent.
So $\mathcal{O}^n = \ker \epsilon \oplus \ker (1-\epsilon)$. We note that $\rank (\ker (\epsilon)) = n -1$ and $\rank (\ker (1 - \epsilon)) = 1$. On $\ker (\epsilon)$ the map $\sigma$ acts as identity while on $\ker (1-\epsilon)$ the map $\sigma$ acts as multiplication by $\lambda$. The result follows.
\end{proof}

\begin{corollary}\label{cor}
   Assume Hypothesis \ref{hypo dvr}. Suppose that $G$ is generated by pseudo-reflections in $GL_n(\mathcal{O}_{\pi})$. Then for $n\geq 2$, the image of $G$ in $GL_n(\mathcal{O}/\pi\mathcal{O})$ is also generated by pseudo-reflections.
\end{corollary}

\begin{proof}
   Follows from Proposition \ref{prop gen} and Lemma \ref{new lemma}.
\end{proof}


Now we give the proof of Theorem \ref{thm}.

\begin{proof}

Let $R=\mathcal{O}[X_1, X_2,\dots , X_n]$ and $S=R^G$. Consider the short exact sequence $$0\rightarrow R\xrightarrow{\pi} R \rightarrow R/\pi R \rightarrow 0.$$ We have the long exact sequence $$0\rightarrow R^G\xrightarrow{\pi} R^G \rightarrow (R/\pi R)^G \xrightarrow{\delta} H^1(G,R)\xrightarrow{\pi} H^1(G,R)\rightarrow H^1(G, R/\pi R)\rightarrow \dots $$


By Theorem 6.5.8 in \cite{weibel1994introduction} we have $r \cdot H^1(G,R)=0$ where $r=|G|$. 
As $r$ is a unit in $\mathcal{O}=(R^G)_0$, we have that it is a unit in $R^G$. Thus $H^1(G,R)=0$. Hence we have the short exact sequence $$0\rightarrow R^G\xrightarrow{\pi} R^G \rightarrow (R/\pi R)^G\rightarrow 0.$$

Thus $R^G/\pi R^G\cong (R/\pi R)^G$. Now $R/\pi R=(\mathcal{O}/\pi \mathcal{O})[X_1,X_2,\dots, X_n]$ and $\mathcal{O}/\pi \mathcal{O}$ is a field. Also by Corollary \ref{cor} we have that $G\subseteq GL_n(\mathcal{O}/\pi \mathcal{O})$ is generated by pseudo-reflections. Thus by Theorem \ref{CST}, $(R/\pi R)^G$ is regular. Let $\mathfrak{M}$ be the maximal homogenous ideal of $R^G$. Since $\pi\not \in \mathfrak{M}^2$. Thus $R^G/\pi R^G$ being regular implies that $S=R^G$ is regular.\end{proof}

\section{Proof of Theorem \ref{main thm} for regular local rings}
\label{reg}

Consider the hypothesis below.

\begin{hypothesis}
    \label{hyp thm reg}

    Assume that $(A,\mathfrak{m})$ is a regular local ring and $K$ is its field of fractions. Consider a finite group $G\subseteq GL_n(A)$. Consider the ring $A[X_1,X_2,\dots, X_n]$ and assume that $G$ acts linearly on the ring (fixing $A$). Assume that $|G|$ is invertible in $A$.
\end{hypothesis}

In this section, we will prove the following.

\begin{theorem}
    \label{thm reg}

Assume Hypothesis \ref{hyp thm reg}. Consider $G\subseteq GL_n(K)$ is generated by pseudo-reflections. Then $(A[X_1,X_2,\dots, X_n])^G$ is regular.

\end{theorem}

Now consider the following.

\begin{hypothesis}\label{hyp reg}

Consider the assumptions in Hypothesis \ref{hyp thm reg}. Let $d=\dim (A)$. Now Theorem \ref{thm reg} for the case $d=0$ (i.e. $A$ is a field) follows from Theorem \ref{CST}. Thus we can assume $d\geq 1$. Let $\mathfrak{m}=(x_1,\dots,x_d)$. Set $x=x_1$ and let $\mathfrak{p}$ be the prime ideal $(x)$.

\end{hypothesis}

We have an analogous result to Proposition \ref{prop gen} for regular local rings. Essentially the same proof goes through but we provide it for the sake of completeness.

\begin{lemma}\label{prop analogue}

Assume Hypothesis \ref{hyp reg}. For $n\in \mathbb{N}$, the natural map $\eta : G\rightarrow GL_n(A/\mathfrak{p})$ is injective.\end{lemma}

\begin{proof}
    We consider elements of $GL_n(A)$ and $GL_n(A/\mathfrak{p})$ as matrices. Let $W\in G$ with $\eta (W)=I$. Then $W=I+Y$ where all entries of $Y$ are in $\mathfrak{p}$. Now $(I+Y)^r=W^r=I$ where $r=|G|$. Thus $$rY+ {r\choose 2} Y^2+\dots +\ {r\choose {r-1}} Y^{r-1} + Y^r=0.$$ Set $$U=rI+\ {r\choose 2} Y+\dots +\ {r\choose {r-1}} Y^{r-2} + Y^{r-1}.$$

    So $YU=0$. The image of $U$ in $GL_n(A/\mathfrak{p})$ is $rI$ which is invertible in $GL_n(A/\mathfrak{p})$ since $r$ is invertible in $A/\mathfrak{p}$. Consider the map $U:A^n\rightarrow A^n$. By Nakayama's Lemma it follows that this map is surjective. Since $A^n$ is noetherian module over $A$, the map is an isomorphism. Thus $U\in GL_n(A)$. As $YU=0$, we get $Y=0$. So $\eta$ is injective.\end{proof}

Now we give the proof of Theorem \ref{thm reg} with notations as in Hypothesis \ref{hyp reg}.
\begin{proof}

We prove the result by induction on $d\geq 1$. Now the result for the case $d=1$ (i.e. $A$ is a DVR) follows from Theorem \ref{thm}. So assume $d\geq 2$ and that the result is known for all regular local rings of dimension $d-1$.\smallskip

 As $A$ is a regular local ring we have that $A_{\mathfrak{p}}$ is a regular local ring of dimension $1$, i.e. a DVR which has $K$ as its field of fractions. Let $\kappa(\mathfrak{p})=A_{\mathfrak{p}}/\mathfrak{p} A_{\mathfrak{p}}$ which is the field of fractions of $A/\mathfrak{p}$. Now $G\subseteq GL_n(K)$ is generated by pseudo-reflections. Thus by Corollary \ref{cor}, we have that the image of $G$ in $GL_n(\kappa(\mathfrak{p}))$ is generated by pseudo-reflections. Note that we have a commutative diagram of ring homomorphisms,

 \[
\begin{tikzcd}
 0 \arrow[r] & A \arrow[d, "\iota"] \arrow[r, "x"] & A \arrow[d, "\iota"] \arrow[r, "\gamma"] & A/\mathfrak{p} \arrow[d, "\mu"] \arrow[r] & 0 \\
  0 \arrow[r] & A_{\mathfrak{p}} \arrow[r, "x"] & A_{\mathfrak{p}} \arrow[r, "\gamma'"] & \kappa(\mathfrak{p}) \ar[r] & 0
\end{tikzcd}
\]

where $\mu$ is the standard ring homomorphism of $A/\mathfrak{p}$ to its field of fractions which 
is injective. As we have the above commutative diagram, we also have the below commutative diagram of group homomorphisms with $\mu$ being injective.

\[
\begin{tikzcd}
GL_n(A) \arrow[d, "\iota"] \arrow[r, "\gamma"] & GL_n(A/\mathfrak{p}) \arrow[d, "\mu"]  \\
    GL_n(A_{\mathfrak{p}}) \arrow[r, "\gamma'"] & GL_n(\kappa(\mathfrak{p}))
\end{tikzcd}
\]

We have that $\gamma|_G : G\rightarrow GL_n(A/\mathfrak{p})$ is injective by Lemma \ref{prop analogue}. Thus we have that $\mu\circ \gamma|_G$ is injective. Now $\gamma'\circ \iota|_G=\mu\circ \gamma|_G$. So the image of $G$ in $GL_n(\kappa(\mathfrak{p}))$ is generated by pseudo-reflections. By induction we have $(A/\mathfrak{p}\ [X_1,X_2,\dots X_n])^G$ is regular. Let $R=A [X_1, X_2,\dots , X_n]$ and $S=R^G$. Consider the short exact sequence $$0\rightarrow R\xrightarrow{x} R \rightarrow R/\mathfrak{p} R \rightarrow 0.$$ We have the long exact sequence $$0\rightarrow R^G\xrightarrow{x} R^G \rightarrow (R/\mathfrak{p} R)^G \xrightarrow{\delta} H^1(G,R)\xrightarrow{x} H^1(G,R)\rightarrow H^1(G, R/\mathfrak{p} R)\rightarrow \dots $$


By Theorem 6.5.8 in \cite{weibel1994introduction} we have $r \cdot H^1(G,R)=0$ where $r=|G|$. 
As $r$ is a unit in $A=(R^G)_0$, we have that it is a unit in $R^G$. Thus $H^1(G,R)=0$. Hence we have the short exact sequence $$0\rightarrow R^G\xrightarrow{x} R^G \rightarrow (R/\mathfrak{p} R)^G\rightarrow 0.$$

Thus $R^G/\mathfrak{p} R^G\cong (R/\mathfrak{p} R)^G$ which is regular. Let $\mathfrak{M}$ be the maximal homogenous ideal of $R^G$. Since $x\not \in \mathfrak{M}^2$. Thus $R^G/\mathfrak{p} R^G$ being regular implies that $S=R^G$ is regular.\end{proof}

\section{Proof of Theorem \ref{main thm} for regular domains} \label{regdom}

Now we complete the proof of Theorem \ref{main thm}.

\begin{proof}
    Now $A$ is a regular domain and $K$ is its field of fractions. Let $P$ be a maximal ideal of $A$. So $(A_P,P A_P)$ is a regular local ring with $K$ as the field of fractions. Also, as $G$ is a subgroup of $GL_n(A)$, it is naturally a subgroup of $GL_n(A_P)$. Further since $|G|$ is invertible in $A$, we have $|G|$ is invertible in $A_P$. Thus for proving Theorem \ref{main thm}, it suffices to prove that $(A_P[X_1, X_2,\dots , X_n])^G$ is regular which follows from Theorem \ref{thm reg}.
\end{proof}

\medskip

\section{Conflict of interest statement}
The authors have no conflict of interest to declare that are relevant to this article.

\section{Data availability statement}

\emph{The manuscript has no associated data.}

\bigskip

\noindent {\it Acknowledgements:} We would like to thank Prof. Gregor Kemper for making us aware of the work of David Mundelius. We also thank the two referees for many pertinent comments.

\medskip

\bibliographystyle{plain}

\end{document}